\documentclass{article}

\usepackage{paulmathdocs}
\usepackage{thm-restate}
\usepackage{authblk}
\newtheoremstyle{pmd-break-label}
  {\topsep}{\topsep}{\normalfont}{}{\bfseries}{.}{\newline}
  {\thmname{#1}\thmnote{ #3}}
\theoremstyle{pmd-break-label}
\newtheorem*{corollary*}{Corollary}
\newtheorem*{theorem*}{Theorem}
\newcommand{\mainTheoremRnStatement}{%
  For every $n \in \bbN$, there exists a subset $A \subseteq \bbR^n$ such that
  $$ \dH(\bbR^n, A) = 1 \quad \text{ and } \quad \dGH(\bbR^n, A) = \sqrt{\frac{n+1}{2n}} . $$
}

\newcommand{\mainTheoremManifoldsStatement}{%
  Let $M$ be a connected Riemannian manifold without boundary of dimension $m \in \bbN$, let $a \in M$, and let $L > 1$.
  \\ Then there exists an $r_0 > 0$ such that for all $r \in (0, r_0)$ we have
  $$ \dGH(M \setminus B_r(a), M) \leq L \cdot \sqrt{\frac{m+1}{2m}} \dH(M \setminus B_r(a), M) \quad \text{ and } \quad \dH(M \setminus B_r(a), M) = r $$
  where $M \setminus B_r(a)$ is equipped with the subspace metric.
}

\title{A converse bound for $\dGH$ vs. $\dH$ for Euclidean space and Riemannian manifolds}
\author{Paul Schott}
\affil{Karlsruher Institut für Technologie}
\date{}

\begin{document}

\maketitle

\begin{abstract}
  We prove for the Euclidean space $\left(\bbR^n, \|.\|_2\right)$ the existence of subsets $A^n \subseteq \bbR^n$ with $\dH(\bbR^n, A^n)$ arbitrarily small, such that $\dGH(\bbR^n, A^n) = \sqrt{\frac{n+1}{2n}}\dH(\bbR^n, A^n)$, showing that the lower bound \linebreak $\dGH(\bbR^n, A^n) \geq \sqrt{\frac{n+1}{2n}}\dH(\bbR^n, A^n)$ proved in \cite{AdamsBogatyiFrickEtAl2026} is tight for sets that are arbitrarily dense in the surrounding space. We also prove a similar statement for Riemannian manifolds and sufficiently dense subsets as a converse to \cite{AdamsFrickMajhiMcBride2023}.
\end{abstract}

\section{Introduction}

  While originally introduced in the 1970s and 1980s as a means of describing convergence of compact metric spaces, the Gromov--Hausdorff distance has found many new applications as a theoretical framework in research areas like shape-matching, see \cite{Schmiedl2017}.
  In its full generality, it is very hard to calculate or closely bound the Gromov--Hausdorff distance from both sides at once, and even for the restrictive class of metric trees, approximating it within a relative factor less than three was proven to be NP-hard in \cite{AgarwalFoxNathEtAl2018}.
  \\ In the situation that we compare a space $X$ and a subspace $A \subseteq X$, we get an upper bound for the Gromov-Hausdorff distance through the related Hausdorff distance. Generally, one easily sees that the Hausdorff distance can become arbitrarily large while the Gromov--Hausdorff distance becomes arbitrarily small. However, in \cite{AdamsFrickMajhiMcBride2023}, \cite{AdamsMajhiManinVirkZava2024MetricGraphs}, or \cite{Schott2026MetricGraphs} the authors proved sufficient criteria on the pairs $(X, A)$ of spaces and their subsets that guarantee converse inequalities of the form
  $$ \dGH(X, A) \geq c_X \dH(X, A) \quad \text{ for } \quad \dH(X, A) \leq r_X $$ 
  for some $c_X, r_X > 0$ depending only on $X$.
  \\ In particular, as a special case of a result from \cite{AdamsBogatyiFrickEtAl2026}, we have:
  \hypertarget{cor:euclidean-lower-bound}{}
  \begin{corollary*}[A {\normalfont(cf. Theorem~3.1 and the remark after Theorem~2.9 in \cite{AdamsBogatyiFrickEtAl2026})}]
    Let $n \in \bbN$ and let $A \subseteq \bbR^n$ satisfy $0 < \dH(\bbR^n, A) < \infty$. Then
    $$
      \frac{\dGH(\bbR^n, A)}{\dH(\bbR^n, A)}
      \geq \sqrt{\frac{n+1}{2n}} .
    $$
  \end{corollary*}
  In \cite{AdamsFrickMajhiMcBride2023}, the authors proved a similar result for compact Riemannian manifolds of bounded sectional curvature:
  \hypertarget{cor:manifold-lower-bound}{}
  \begin{corollary*}[B {\normalfont(cf. Theorem~4 in \cite{AdamsFrickMajhiMcBride2023})}]
    Let $M$ be a connected, compact Riemannian manifold without boundary of dimension $n$, with sectional curvature bounded above by $\kappa \in \bbR$.
    Then there exists $\varepsilon_M>0$ such that, for every $A \subseteq M$ with $0 < \dH(A,M) < \varepsilon_M$,
    $$
      \frac{\dGH(A,M)}{\dH(A,M)} \geq \alpha(n,\kappa),
    $$
    where
    $$
      \alpha(n,\kappa) :=
      \begin{cases}
        \sqrt{\frac{n+1}{2n}}, & \kappa \leq 0,\\[1.2em]
        \sqrt{\frac{n+1}{2n}}
        \dfrac{\sin\!\left(\frac{\pi}{2}\sqrt{\frac{\kappa}{\kappa+1}}\right)}
        {\frac{\pi}{2}\sqrt{\frac{\kappa}{\kappa+1}}}, & \kappa > 0.
      \end{cases}
    $$
  \end{corollary*}
  A natural question to ask is whether these constants simply arise from the chosen arguments for the proof, or whether they are optimal. By constructing counterexamples, we show such optimality for Euclidean space and manifolds with nonpositive sectional curvature in the form of the following theorems:
  \begin{theorem*}[\ref{thm:MainTheoremRn}]
    \mainTheoremRnStatement
  \end{theorem*}

  For the case of Riemannian manifolds we can slightly alter this construction, giving in particular tightness of the bound from \hyperlink{cor:manifold-lower-bound}{Corollary~B} for the case of manifolds of nonpositive sectional curvature.

  \begin{theorem*}[\ref{thm:MainTheoremManifolds}]
    \mainTheoremManifoldsStatement
  \end{theorem*}

  Since in the case of general Riemannian manifolds and the particular case of the Euclidean space $\bbR^n$ we are potentially working with noncompact spaces and subsets, we start by introducing tools to bound the distortion of correspondences between such spaces.
  We then introduce a barycentric coordinate system for $\bbR^n$ as a more natural way of viewing Euclidean space for our purposes.

\section*{Disclaimer}
  In the generation of the latex code for this paper generative AI (specifically OpenAI's Codex) was used in limited capacity and controlled by the author for the formatting of the document and the generation of the illustrative graphic.
\section*{Acknowledgements}
  The author would like to thank Alexander Lytchak for suggesting the problem and related literature on the subject of bounding Gromov-Hausdorff distances.

\section{Preliminaries}
  
  \subsection{Correspondences}
    We define a correspondence between two metric spaces $X$ and $Y$ as a relation $R \subseteq X \times Y$ such that every element in either set stands in relation to an element of the other set.
    \\ For a relation between metric spaces we define
    \begin{definition}
      Let $X$ and $Y$ be two metric spaces and $R \subseteq X \times Y$ a relation between them. 
      \begin{enumerate}[(i)]
        \item We define the distortion of $R$ as
        $$ \dis(R) := \sup\left\{\left|d_X(x, x') - d_Y(y, y')\right| : (x, y), (x', y') \in R\right\} . $$
      \end{enumerate}
      Let now $Y$ be a subset of $X$ equipped with the subspace metric, and let $R \subseteq X \times Y$ be a nonempty relation.
      \begin{enumerate}[(i)] \setcounter{enumi}{1}
        \item We define the motion of $R$ as
        $$ \operatorname{Mot}(R) := \sup\left\{d_X(x, y) : (x, y) \in R\right\} . $$
        \item We define the fixed points of $R$ as
        $$ \operatorname{Fix}_X(R) := \left\{x \in X : R[x] = \{x\}\right\} . $$
        \item We define the essential part of $R$ as the relation
        $$ \tilde{R} := R \cap \left(\left(X \setminus \operatorname{Fix}_X(R)\right) \times Y\right) $$
      \end{enumerate}
    \end{definition}
    \begin{lemma} \label{lem:boundingcorrespondencesthroughmotionandnonfixpointcorrespondences}
      Let $X$ be a metric space and $Y \subsetneq X$ be a subset equipped with the subspace metric.
      \\ Let $R \subseteq X \times Y$ be a correspondence and let $\tilde{R}$ be its essential part.
      \\ Then $\tilde{R}$ is nonempty and we have
      $$ \dis(R) \leq \max\left\{\operatorname{Mot}(\tilde{R}), \dis\left(\tilde{R} \right) \right\}. $$
    \end{lemma}
    \begin{proof}
      Consider pairs $(x, y), (x', y') \in R$.
      \\ \textbf{Case 1)} $R[x] = \{x\}$, $R[x'] = \{x'\}$.
      \\ Then $y = x$ and $y' = x'$ and therefore,
      $$ \left|d(x, x') - d(y, y')\right| = \left|d(x, x') - d(x, x')\right| = 0 . $$
      \textbf{Case 2)} $R[x] = \{x\}$ and $x' \in X \setminus \operatorname{Fix}_X(R)$.
      $$ \left|d(x, x') - d(y, y')\right| \leq d(x, y) + d(x', y') = d(x', y') \leq \operatorname{Mot}(\tilde{R}) . $$
      \\ The analogous argument works if the roles of $(x, y)$ and $(x', y')$ are reversed.
      \\ \textbf{Case 3)} Neither $R[x] = \{x\}$ nor $R[x'] = \{x'\}$
      \\ Then we have $(x, y), (x', y') \in \tilde{R}$, implying
      $$ \left|d(x, x') - d(y, y')\right| \leq \dis(\tilde{R}) . $$
    \end{proof}
  
  \subsection{Convention for the Gromov--Hausdorff distance for noncompact spaces}
    When working with noncompact spaces, one often compares pointed spaces with a notion of pointed Gromov--Hausdorff distance.
    \\ In this paper we do no such thing and use the standard two equivalent definitions:
    \begin{definition}
      Let $(X, d_X)$ and $(Y, d_Y)$ be two (possibly non-compact) metric spaces.
      \\ We define the Gromov--Hausdorff distance $\dGH((X, d_X), (Y, d_Y))$ equivalently via the two expressions
      $$ \inf\left\{\dis(R) : R \text{ is a correspondence between } X \text{ and } Y \right\} . $$
      and
      $$ \inf\left\{\dH^Z(X', Y') : Z \text{ metric space and } X', Y' \text{ images of isometric embeddings of } X, Y \text{ into } Z \right\} . $$
    \end{definition}

  \subsection{Bilipschitz approximations}
    \begin{notation}[Bilipschitz approximate equality]
      For two numbers $\alpha, \beta \in [0, \infty)$ and $L \geq 1$ we write
      $$ \alpha \overset{L}{\approx} \beta \quad \text{ if } \quad \frac{1}{L} \alpha \leq \beta \leq L \alpha . $$
    \end{notation}
    \begin{definition}\label{def:bilipschitz-map}
      For two metric spaces $X$ and $Y$, a map $f : X \to Y$ and a constant $L \geq 1$ we say that $f$ is an \emph{$L$-bilipschitz map} if
      $$ d_Y(f(x), f(x')) \overset{L}{\approx} d_X(x, x') \quad \text{ for all } x, x' \in X . $$
    \end{definition}
    \begin{lemma} \label{lem:uniformapproximationforboundedsetsandBilipschitzmaps}
      Let $\alpha, \beta \in [0, \infty)$ and $L > 1$.
      \\ For $\alpha', \beta' \in [0, \infty)$ with $\alpha \overset{L}{\approx} \alpha'$ and $\beta \overset{L}{\approx} \beta'$ we have
      $$ |\alpha' - \beta'| \leq L |\alpha - \beta| + (L - L^{-1}) \max\{\alpha, \beta\} . $$
    \end{lemma}
    \begin{proof}
      $$ \alpha' - \beta' \leq L \alpha - L^{-1} \beta \leq L |\alpha - \beta| + (L - L^{-1}) \beta . $$
      $$ \beta' - \alpha' \leq L \beta - L^{-1} \alpha \leq L |\alpha - \beta| + (L - L^{-1}) \alpha . $$ 
    \end{proof}
  
  \subsection{Regular simplices and barycentric coordinates}
    \subsubsection{Definition}
      By Jung's theorem for $\bbR^n$ (see \cite[Theorem~2.6]{DanzerGrunbaumKlee1963}), every subset of $\bbR^n$ with circumradius $1$ has diameter at least
      $$ l_n := \sqrt{\frac{2(n+1)}{n}} . $$
      Equality implies that the closure of the set contains the vertices of a regular $n$-simplex of edge length $l_n$. In particular, up to isometry, the unique $(n+1)$-point set attaining equality is the vertex set $V_n := \left\{v_0, v_1, \ldots, v_n\right\}$ of a regular $n$-simplex, which for our purposes we choose to be inscribed into the unit sphere $S^{n-1} \subseteq \bbR^n$.
      We now fix a choice of such a set $V_n$ and establish properties for it.
      \begin{lemma}
        \begin{equation}
          \angles{v_i, v_j} = 
          \begin{cases}
            1 & \text{ for } i = j \\
            \frac{-1}{n} & \text{ for } i \neq j
          \end{cases}
        \end{equation}
      \end{lemma}
      \begin{proof}
        The linear isometries of the surrounding space $\bbR^n$ contain a subgroup of isometries restricting to the permutations of $V_n$.
        \\ In particular, this gives us that there exists an $\alpha \in \bbR$ such that 
        $$ \angles{v_i, v_j} = 
          \begin{cases}
            1 & \text{ for } i = j \\
            \alpha & \text{ for } i \neq j
          \end{cases} $$
        and using the transitivity of isometries, we see that
        $$ 0 = \angles{0, 0} = \angles{\sum_{i = 0}^n v_i, \sum_{i = 0}^n v_i} = \sum_{i = 0}^n \angles{v_i, v_i} + \sum_{\substack{i, j = 0 \\ i \neq j}}^n \angles{v_i, v_j} = (n+1) + n(n+1) \alpha $$
        implying that $\alpha = \frac{-1}{n}$.
      \end{proof}
    \subsubsection{Barycentric coordinates}
      \begin{definition}
        The $n+1$ vectors in $V_n$ give us a barycentric coordinate system for $\bbR^n$, meaning that every vector $x \in \bbR^n$ can be uniquely written as a linear combination
        \begin{equation}
          x = \sum_{i = 0}^n x_i v_i \quad \text{for } \quad x_0, \ldots, x_n \in \bbR \quad \text{ with } \quad \sum_{i = 0}^n x_i = 1 .
        \end{equation}
      \end{definition}
      Note that throughout this paper, we never use Euclidean coordinates and always abbreviate for a point $x \in \bbR^n$ its affine coordinates as $x_0, \ldots, x_n$.
      \begin{lemma}
        Let $x \in \bbR^n$.
        The affine coefficients follow the relation
        \begin{equation}
          \angles{x, v_k} = \frac{n+1}{n} x_k - \frac{1}{n} \qquad \text{i.e.} \qquad x_k = \frac{n \angles{x, v_k} + 1}{n+1} .
        \end{equation}
      \end{lemma}
      \begin{proof}
        $$ 
          \angles{x, v_k} = \sum_{i = 0}^n x_i \angles{v_i, v_k} = x_k - \frac{1}{n} \sum_{\substack{i = 0 \\ i \neq k}}^n x_i = x_k - \frac{1}{n}(1 - x_k) = \frac{n+1}{n} x_k - \frac{1}{n} .  
        $$
      \end{proof}
    \subsubsection{A Voronoi partitioning}
      We now divide the space $\bbR^n$ into $n+1$ closed Voronoi cells $P_k$ around the points $v_i$, namely
      \[
        P_k := \left\{x \in \bbR^n : \|x - v_k\| = \min_{i = 0, \ldots, n} \|x - v_i\|\right\} .
      \]
      As a special case for $r \in (0, \infty]$ we partition $B_r(0)$ into cells $P_k^{r} := P_k \cap B_r(0)$.
      \\ In affine coordinates we can express these cells as follows:
      \begin{lemma} \label{lem:VoronoiPartitionInAffineCoordinates}
        For $k = 0, 1, \ldots, n$ we have
        \begin{equation}
          P_k = \left\{\sum_{i = 0}^n x_i v_i : \sum_{i = 0}^n x_i = 1, \text{ and } x_k = \max_{i = 0, \ldots, n} x_i \right\} .
        \end{equation}
        and if $x \in P_k$, we have $x_k \geq \frac{1}{n+1}$ and $\angles{x, v_k} \geq 0$.
      \end{lemma}
      \begin{proof}
        $$ \|x - v_k\|^2 = \|x\|^2 + \|v_k\|^2 - 2\angles{x,v_k}
        = \|x\|^2 + 1 - 2\angles{x,v_k} . $$
        Thus, $x \in P_k$ if and only if $\angles{x,v_k}$ is maximal among all values $\angles{x,v_i}$.
        Since $\angles{x,v_i} = \frac{n+1}{n}x_i - \frac{1}{n}$ is strictly increasing as a function of $x_i$, this is equivalent to $x_k$ being maximal among all $x_i$, giving us the desired description of the Voronoi cells.
        \\ We conclude the lower bounds based on the sum formulas
        $$ \sum_{i = 0}^n x_i = 1 \quad \text{ and } \quad \sum_{i = 0}^n \angles{x, v_i} = 0 . $$
      \end{proof}
  
  \subsection{Riemannian manifolds}
      For the definition of a Riemannian manifold, see \cite{doCarmo1992}.
      \\ For a connected Riemannian manifold $(M, g)$ we denote the induced Riemannian distance function by $d_g : M \times M \to [0, \infty)$, or simply $d$ for short.
      
      For a Riemannian manifold $(M, g)$ and a point $a \in M$ we denote the exponential map at $a$ by 
      $$ \exp_a : \mathcal{D}_a \subseteq T_a M \to M . $$
      where $\mathcal{D}_a$ is its domain at $a$ (equipped with the Euclidean metric induced by $g_a$).
      \\ We note that its restriction to a ball $B_r(0_{T_a M}) \subseteq \mathcal{D}_a$ for sufficiently small $r > 0$ is a diffeomorphism onto its image $B_r(a) \subseteq M$. Its inverse is called the normal chart $\operatorname{norm}_a : B_r(a) \to B_r(0_{T_a M})$.
      \\ The maps preserve distance to the origin in a sufficiently small neighborhood around $a$, by which we mean that for sufficiently small $r > 0$,
      $$ d_{g_a}(0_{T_a M}, v) = d_g(a, \exp_a(v)) \quad \text{ for } v \in B_{r}(0_{T_a M}) \subseteq \mathcal{D}_a. $$
      We also note that for every $L > 1$, there exists a sufficiently small $r := r(a, L) > 0$ such that both 
      $$ \exp_a : \bar{B}_{r}(0_{T_a M}) \to \bar{B}_{r}(a) \quad \text{ and } \quad \operatorname{norm}_a : \bar{B}_{r}(a) \to \bar{B}_{r}(0_{T_a M}) $$
      are $L$-bilipschitz maps. This is a direct consequence of the fact that the differential of the exponential map at $0_{T_a M}$ is the identity map on $T_a M$.

\section{Statement and Proof}
  \subsection{The case of Euclidean space}
      \begin{restatable}{theorem}{mainTheoremRn}\label{thm:MainTheoremRn}
        \mainTheoremRnStatement
      \end{restatable}
      
      We first prove two estimates that will be used to bound the distortion of the correspondence constructed in the proof.

      \begin{lemma} \label{lem:MotionOfInnerCirclePoints}
        For $x \in P_k^{1}$ we have
        $$ \|x - v_k\| \leq \sqrt{2} . $$
      \end{lemma}
      \begin{proof}
        As $x \in P_k$ implies that $\angles{x, v_k} \geq 0$ we have:
        $$ \|x - v_k\|^2 = \underbrace{\|x\|^2}_{\leq 1} + \underbrace{\|v_k\|^2}_{= 1} - 2 \underbrace{\angles{x, v_k}}_{\geq 0} \leq 2 . $$
      \end{proof}
  
      \begin{lemma} \label{lem:DistortionForAnnulusAndInnerCircle}
        For $k = 0, \ldots, n$ and points $x \in P_k^{\sfrac{\sqrt{2}}{2}}$ and $x' \in \bar{B}_1(0) \setminus B_{\sfrac{\sqrt{2}}{2}}(0)$ we have:
        \begin{equation}
          \left|\left\|x - x'\right\| - \left\| v_k - \frac{x'}{\|x'\|}\right\|\right| \leq \sqrt{2} . 
        \end{equation}
      \end{lemma}
      \begin{proof}
        We define
        \begin{align*}
          F(x, x') & = \left\|x - x'\right\| ,
          \\ G(x, x') = G(x') & = \left\|v_k - \frac{x'}{\|x'\|}\right\| .
        \end{align*}
        Within the specified domains we now show that
        $$ \left|F(x, x') - G(x')\right| \leq \sqrt{2} . $$ 

        Write $\rho := \|x'\| \in \left[\frac{\sqrt{2}}{2}, 1\right]$ and $y' := \frac{x'}{\|x'\|}$.
        \\ As we proved in \cref{lem:VoronoiPartitionInAffineCoordinates}, $x \in P_k^{\sfrac{\sqrt{2}}{2}}$ implies that $\angles{x, v_k} \geq 0$.
        
        \textbf{Part 1)} $F - G \leq \sqrt{2}$
        
        Using $\rho^2 \leq 1$ and $\|x\|^2 \leq \frac{1}{2}$, we obtain
        \begin{flalign*}
          & F(x, x') - G(x')
          \\ = & \|x - \rho y'\| - \|v_k - y'\|
          \\ \leq & \|x - \rho y'\| - \rho\|v_k - y'\|
          \\ \leq & \|x - \rho y'\| - \|\rho v_k - \rho y'\|
          \\ \leq & \|x - \rho v_k\|
          \\ = & \sqrt{\|x\|^2 + \rho^2 - 2\rho\angles{x, v_k}}
          \\ \leq & \sqrt{\frac{1}{2} + 1}
          \\ \leq & \sqrt{2} .
        \end{flalign*}
        
        \textbf{Part 2)} $G - F \leq \sqrt{2}$:
        \\ For the reverse difference, put $t := \angles{v_k, -y'}$ and note that $t \in [-1, 1]$. 
        \\ Generally we have

        $$ G(x') = \sqrt{\|v_k - y'\|^2} = \sqrt{\|v_k\|^2 + \|y'\|^2 - 2 \angles{v_k, y'}} = \sqrt{2 + 2t} . $$

        \textbf{Case 1)} $t \leq 0$
        \\ In this case, the inequality follows from
        $$ G(x') - F(x, x') \leq G(x') = \sqrt{2 + 2t} \leq \sqrt{2} . $$
        
        \textbf{Case 2)} $t > 0$
        \\ In this case, using that $\rho \geq \frac{\sqrt{2}}{2}$, we have
        $$ F(x, x') = \|x - \rho y'\| \geq \angles{x - \rho y', v_k} = \underbrace{\angles{x, v_k}}_{\geq 0} + \underbrace{\rho\angles{-y', v_k}}_{= \rho t} \geq \rho t \geq \frac{\sqrt{2}}{2}t. $$
        From this we conclude 
        $$ G(x') - F(x, x') \leq \sqrt{2 + 2t} - \frac{\sqrt{2}}{2} t =: H(t) . $$
        We see that $H(0) = \sqrt{2}$ and for $t \in [0, 1]$ we have
        $$ H'(t) = \frac{1}{\sqrt{2 + 2t}} - \frac{\sqrt{2}}{2} = \frac{1 - \sqrt{1+t}}{\sqrt{2 + 2t}} \leq 0 . $$
        From this we can conclude that $H(t) \leq H(0) = \sqrt{2}$ on $[0, 1]$, proving the desired inequality.
      \end{proof}
      \begin{proof}[Proof of \cref{thm:MainTheoremRn}]
        Set $A^n := \bbR^n \setminus B_1(0)$. Every point $x \in B_1(0)$ has distance $1 - \|x\|$ from $A^n$, so $\dH(\bbR^n, A^n) = 1$.
        We define a relation $R \subseteq \bbR^n \times A^n$ by
        $$ R[x] =
          \begin{cases}
            \left\{x\right\} & \text{ for } x \in A^n ,
            \\ \left\{\frac{x}{\|x\|}\right\} & \text{ for } x \in B_1(0) \setminus B_{\sfrac{\sqrt{2}}{2}}(0) ,
            \\ \left\{v_k : k = 0, \ldots, n \text{ such that } x \in P_k^{\sfrac{\sqrt{2}}{2}}\right\} & \text{ for } x \in B_{\sfrac{\sqrt{2}}{2}}(0) .
          \end{cases}
        $$
        
        \begin{figure}[H]
        \centering
        \begin{subfigure}[t]{0.47\textwidth}
          \centering
          \resizebox{\linewidth}{!}{%
            \begin{tikzpicture}[scale=1]

              \colorlet{outsideDarkGray}{black!70}
              \path[use as bounding box] (-2.7, -2.7) rectangle (2.7, 2.7);
              \clip (-2.7, -2.7) rectangle (2.7, 2.7);

              \node[circle, fill=outsideDarkGray, inner sep=0.75pt, label=right:$v_0$] (v1) at ({cos(0)}, {sin(0)}) {};
              \node[circle, fill=outsideDarkGray, inner sep=0.75pt, label=above left:$v_1$] (v2) at ({cos(120)}, {sin(120)}) {};
              \node[circle, fill=outsideDarkGray, inner sep=0.75pt, label=below left:$v_2$] (v3) at ({cos(240)}, {sin(240)}) {};
              \node[circle, fill=black!25, inner sep=0.75pt] (origin) at (0, 0) {};

              \begin{scope}[on background layer]
                \path[fill=outsideDarkGray, fill opacity=0.6, even odd rule]
                  (-2.5, -2.5) rectangle (2.5, 2.5)
                  (0, 0) circle (1cm);

                \path[fill=magenta, fill opacity=0.1, even odd rule]
                  (0, 0) circle (1cm)
                  (0, 0) circle[radius={sqrt(2)/2}];

                \draw[gray, dotted] (v1) -- (v2);
                \draw[gray, dotted] (v2) -- (v3);
                \draw[gray, dotted] (v3) -- (v1);

                \begin{scope}[transparency group, opacity=0.35]
                  \fill[red] (0, 0) -- (-60:{sqrt(2)/2})
                    arc[start angle=-60, end angle=60, radius={sqrt(2)/2}] -- cycle;
                  \fill[green] (0, 0) -- (60:{sqrt(2)/2})
                    arc[start angle=60, end angle=180, radius={sqrt(2)/2}] -- cycle;
                  \fill[blue] (0, 0) -- (180:{sqrt(2)/2})
                    arc[start angle=180, end angle=300, radius={sqrt(2)/2}] -- cycle;
                \end{scope}

                \foreach \angle in {0,5,...,355} {
                  \draw[magenta, thin, opacity=0.75] ({cos(\angle)}, {sin(\angle)}) -- ({(sqrt(2)/2)*cos(\angle)}, {(sqrt(2)/2)*sin(\angle)});
                }

                \draw[black] (0, 0) circle (1cm);
              \end{scope}

            \end{tikzpicture}
          }
        \end{subfigure}
        \caption{The sets $R[x]$ in the case $n = 2$}
      \end{figure}
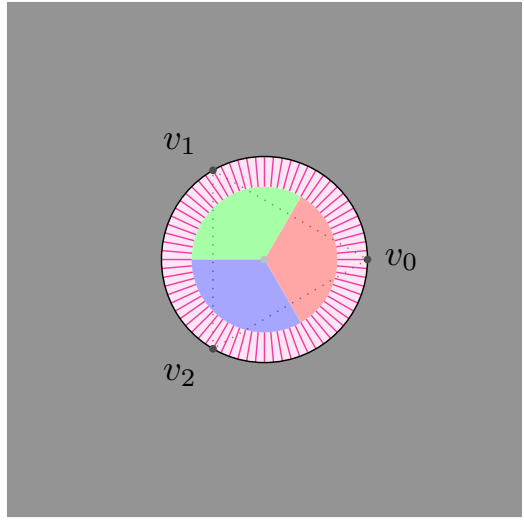
        $R$ by construction assings to every point in $\bbR^n$ a point in $A^n$ and conversely contains $(x, x)$ for all $x \in A^n$, making it a correspondence. In particular we have 
        $$ \operatorname{Fix}_{\bbR^n}(R) = A^n . $$
        and we accordingly define the essential part $\tilde{R} = R \cap \left(B_1(0) \times \bbR^n\right)$.
        \\ By \cref{lem:boundingcorrespondencesthroughmotionandnonfixpointcorrespondences} we now see that
        \begin{equation} \label{eq:distortionviaMotionAndEssentialPartForEuclideanSpace}
          \dis(R) \leq \max\left\{\dis(\tilde{R}), \operatorname{Mot}(\tilde{R})\right\} .
        \end{equation}
        We now examine $\|x-y\|$ for $(x,y) \in \tilde{R}$ and note that
        \begin{equation} \label{eq:motionofthecorrespondenceEuclideanSpace}
          \|x - y\| 
          \begin{cases} 
            \in \left[0, 1 - \frac{\sqrt{2}}{2}\right] \subseteq [0, \frac{\sqrt{2}}{2}] & \text{ for } x \in B_1(0) \setminus B_{\sfrac{\sqrt{2}}{2}}(0) ,
            \\ \in \left[0, \sqrt{2}\right] & \text{ for } x \in B_{\sfrac{\sqrt{2}}{2}}(0) .
          \end{cases}
        \end{equation}
        wherein we used \cref{lem:MotionOfInnerCirclePoints}.
        \\ We can conclude that $\operatorname{Mot}(\tilde{R}) \leq \sqrt{2} \leq l_n$ and by \eqref{eq:distortionviaMotionAndEssentialPartForEuclideanSpace} are left to examine $\dis(\tilde{R})$.
        \\ We now fix pairs $(x, y), (x', y') \in \tilde{R}$ to examine the correspondence's distortion and w.l.o.g. assume $\|x\| \leq \|x'\|$ for simplicity. We distinguish the following all-encompassing cases:
        
        \textbf{Case 1)} $x, x' \in B_1(0) \setminus B_{\sfrac{\sqrt{2}}{2}}(0)$
        \\ By \eqref{eq:motionofthecorrespondenceEuclideanSpace} we get
        $$ \left|\left\|x - x'\right\| - \left\|y - y'\right\|\right| \leq \left\|x - y\right\| + \left\|x' - y'\right\| \leq \frac{\sqrt{2}}{2} + \frac{\sqrt{2}}{2} \leq \sqrt{2} \leq l_n . $$
        
        \textbf{Case 2)} $x \in B_{\sfrac{\sqrt{2}}{2}}(0)$ and $x' \in B_1(0) \setminus B_{\sfrac{\sqrt{2}}{2}}(0)$
        \\ We are then in the situation of \cref{lem:DistortionForAnnulusAndInnerCircle} with $x \in P_k^{\sfrac{\sqrt{2}}{2}}$ and $y = v_k$ for some $k$ and as we have seen, we have
        $$ \left|\left\|x - x'\right\| - \left\|y - y'\right\|\right| = \left|\left\|x - x'\right\| - \left\|v_k - \frac{x'}{\|x'\|}\right\|\right| \leq \sqrt{2} \leq l_n $$
        Since we assumed $\|x\| \leq \|x'\|$, if we are not in cases 1 or 2, we must have $x, x' \in B_{\sfrac{\sqrt{2}}{2}}(0)$.

        \textbf{Case 3)} $x, x' \in B_{\sfrac{\sqrt{2}}{2}}(0)$ and $y = y'$
        \\ In this case we have 
        $$ \left|\left\|x - x'\right\| - \left\|y - y'\right\|\right| = \left\|x - x'\right\| \leq \diam\left(B_{\sfrac{\sqrt{2}}{2}}(0)\right) = \sqrt{2} \leq l_n . $$

        \textbf{Case 4)} $x, x' \in B_{\sfrac{\sqrt{2}}{2}}(0)$ and $y \neq y'$.
        \\ The distance of two different vertices $y = v_{k}$ and $y' = v_{k'}$ is given by $l_n$.
        \\ Since $\|x - x'\| \in [0, 2] \subseteq [0, 2 l_n]$, we have
        $$ \left|\left\|x - x'\right\| - \left\|y - y'\right\|\right| \leq l_n . $$
        Note that in the particular case $x = x' = 0$ and $y \neq y'$, the value $l_n$ is assumed, showing $\dis(R) = l_n$.
        Therefore,
        $$ \dGH(\bbR^n, A^n) \leq \frac{1}{2}\dis(R) = \sqrt{\frac{n+1}{2n}} . $$
        Since $\dH(\bbR^n, A^n) = 1$, \hyperlink{cor:euclidean-lower-bound}{Corollary~A} gives the reverse inequality. Thus equality holds, proving the theorem with $A = A^n$.
      \end{proof}
      \begin{corollary} \label{cor:MainTheoremRnArbitrarilySmall}
        For $n \in \bbN$ and $\varepsilon > 0$, let $A^n := \bbR^n \setminus B_1(0)$ be the set constructed in the proof of \cref{thm:MainTheoremRn}. The set $\varepsilon \cdot A^n \subseteq \bbR^n$ satisfies
        $$ \dGH(\bbR^n, \varepsilon \cdot A^n) = \frac{l_n}{2} \dH(\bbR^n, \varepsilon \cdot A^n) \quad \text{ as well as } \quad \dH(\bbR^n, \varepsilon \cdot A^n) \leq \varepsilon . $$
      \end{corollary}
      \begin{proof}
        Since $\bbR^n$ is invariant under dilation, we see that
        $$ \dGH(\bbR^n, \varepsilon \cdot A^n) = \dGH(\varepsilon \cdot \bbR^n, \varepsilon \cdot A^n) = \varepsilon \cdot \dGH(\bbR^n, A^n) , $$
        and in a similar way, 
        $$ \dH(\bbR^n, \varepsilon \cdot A^n) = \varepsilon \cdot \dH(\bbR^n, A^n) . $$
      \end{proof}
      \begin{remark} \label{rem:correspondenceForArbitrarilySmallComplementsInRn}
        We shall denote the correspondence that results from rescaling the correspondence $R$ constructed in the proof of \cref{thm:MainTheoremRn} by a factor of $r > 0$ as $R_r$. We shall refer to it again in the proof of \cref{thm:MainTheoremManifolds}.
      \end{remark}
  \subsection{The case of manifolds}
    \begin{restatable}{theorem}{mainTheoremManifolds}\label{thm:MainTheoremManifolds}
      \mainTheoremManifoldsStatement
    \end{restatable}
    \begin{proof}[Proof of \cref{thm:MainTheoremManifolds}]
      \textbf{Part 1)}
      \\ Note that for $r_1 > 0$ sufficiently small and $r \in (0, r_1)$, the exponential map $\exp_a$ is a diffeomorphism from $B_r(0_{T_a M})$ onto its image $B_r(a)$ and in particular distance-preserving along radial lines emanating from $a$.
      This in particular gives us
      $$ \dH(M, M \setminus B_r(a)) = r . $$
      
      \textbf{Part 2)}
      \\ Recall for $r \in (0, r_1)$ the correspondence $R_r \subseteq \bbR^m \times (r \cdot A^m)$ discussed in \cref{rem:correspondenceForArbitrarilySmallComplementsInRn}.
      \\ We now identify $T_a M$ equipped with the Riemannian inner product $g_a$ at $a$ with the Euclidean space $\bbR^m$, allowing us for $r \in (0, r_1)$ to define a correspondence
      $$ R_r^M[p] = 
      \begin{cases}
        \exp_a(R_r[\operatorname{norm}_a(p)]) & \text{ for } p \in B_{r}(a)
        \\ \{p\} & \text{ for } p \in M \setminus B_r(a) .
      \end{cases} $$
      We see that it fixes the points
      $$ \operatorname{Fix}_M(R_r^M) = M \setminus B_r(a) . $$
      and define accordingly its essential part
      $$ \tilde{R}_r^M := R_r^M \cap \left(B_r(a) \times M\right) . $$
      For $(p, q) \in B_r(a) \times \bar{B}_r(a)$ we have by design of our correspondence
      $$ (p, q) \in R_r^M \quad \text{ if and only if } \quad (\operatorname{norm}_a(p), \operatorname{norm}_a(q)) \in R_r . $$
      
      \textbf{Part 3)}
      \\ We know that if we choose $r_2 \in (0, r_1)$ small enough and $r \in (0, r_2)$, $\operatorname{norm}_a$ restricts to an $L$-bilipschitz map on $\bar{B}_r(a)$ and we see that for $(p, q) \in \tilde{R}_r^M$,
      $$ d(p, q) \leq L d(\operatorname{norm}_a(p), \operatorname{norm}_a(q)) \leq L \operatorname{Mot}(\tilde{R}_r) . $$
      giving us in the supremum that
      \begin{equation} \label{eq:boundforMotionForManifoldCorrespondence}
        \operatorname{Mot}(\tilde{R}_r^M) \leq L \operatorname{Mot}(\tilde{R}_r) \leq r L l_m .
      \end{equation}

      \textbf{Part 4)}
      Since
      $$ \lim_{K \searrow 1} K l_m + 2\left(K - K^{-1}\right) = l_m < L l_m , $$ 
      we can choose a $K > 1$ such that 
      $$ K l_m + 2 \left(K - K^{-1}\right) \leq L l_m . $$
      We now choose $r_0 \in (0, r_2)$ so small that for $r \in (0, r_0)$ the restriction of $\operatorname{norm}_a$ to $\bar{B}_r(a)$ is $K$-bilipschitz.
      \\ By \cref{lem:boundingcorrespondencesthroughmotionandnonfixpointcorrespondences} and \eqref{eq:boundforMotionForManifoldCorrespondence} we are only left to show $\dis(\tilde{R}_r^M) \leq L \cdot l_m r$.
      \\ For $(p, q), (p', q') \in \tilde{R}_r^M$ we have
      $$ d(p, p') \overset{K}{\approx} d(\operatorname{norm}_a(p), \operatorname{norm}_a(p')) \quad \text{ and } d(q, q') \overset{K}{\approx} d(\operatorname{norm}_a(q), \operatorname{norm}_a(q')) $$
      as well as
      $$ d(p, p'), d(q, q') \leq \diam(\bar{B}_r(a)) \leq 2r . $$
      Since $\operatorname{norm}_a$ is $K$-Lipschitz on $B_r(a)$, and by \cref{lem:uniformapproximationforboundedsetsandBilipschitzmaps} we have
      \begin{flalign*}
        & \left|d(p, p') - d(q, q')\right|
        \\ \leq & K \cdot \left|d(\operatorname{norm}_a(p), \operatorname{norm}_a(p')) - d(\operatorname{norm}_a(q), \operatorname{norm}_a(q'))\right| 
        \\ & \quad + (K - K^{-1})\max\left\{d(\operatorname{norm}_a(p), \operatorname{norm}_a(p')), d(\operatorname{norm}_a(q), \operatorname{norm}_a(q'))\right\}
        \\ \leq & K \dis(R_r) + (K - K^{-1}) 2r 
        \\ = & r \cdot \left(K l_m + 2\left(K - K^{-1}\right)\right)
        \\ \leq & L l_m r .
      \end{flalign*}
      Taking the supremum over all pairs $(p, q), (p', q') \in \tilde{R}_r^M$ gives $\dis(\tilde{R}_r^M) \leq L l_m r$, which by \cref{lem:boundingcorrespondencesthroughmotionandnonfixpointcorrespondences} and \eqref{eq:boundforMotionForManifoldCorrespondence} completes the proof.
    \end{proof}
    \begin{remark}[Manifolds of lesser regularity]
      While the positive results use curvature and thus assume at least $C^2$-regularity of the Riemannian metric, the central assumption in the proof of \cref{thm:MainTheoremManifolds} is the existence of a bilipschitz map with constant arbitrarily close to $1$ from a small neighborhood of a point in the manifold to a neighborhood in Euclidean space.
      \\ For this, mere continuity of the Riemannian metric is sufficient, as this continuity implies that any point has a neighborhood and on it a chart whose differential is $L$-bilipschitz on the length of tangent vectors, which implies composing with the chart is $L$-bilipschitz on the length of curves, in particular the length of shortest curves between two points in a potentially smaller neighborhood and so after restriction, our chart is $L$-bilipschitz in the sense of \cref{def:bilipschitz-map}.
    \end{remark}

\let\standardbibitem\bibitem
\renewcommand{\bibitem}[2][]{%
  \def\currentbibkey{#2}%
  \def\schmiedlbibkey{Schmiedl2017}%
  \ifx\currentbibkey\schmiedlbibkey
    \standardbibitem[Schm17]{#2}%
  \else
    \standardbibitem[#1]{#2}%
  \fi
}
\bibliographystyle{alpha}
\bibliography{Main}

\end{document}